\documentclass{article}

\usepackage[utf8]{inputenc}
\usepackage[english]{babel}
\usepackage{amsthm}
\usepackage{amsmath}
\usepackage{amssymb}
\usepackage{url}
\usepackage{enumerate}
\usepackage{color}

\theoremstyle{plain}
\newtheorem{theorem}{Theorem}
\newtheorem{lemma}{Lemma}
\newtheorem{prop}[lemma]{Proposition}
\newtheorem{cor}[lemma]{Corollary}

\theoremstyle{definition}

\theoremstyle{remark}

\newtheorem{example}{Example}

\newcommand{\F}{\mathbb{F}}

\renewcommand{\S}{\mathcal{S}}

\title{On the expansion complexity of sequences over finite fields}
\author{
Domingo G\'omez-P\'erez, L\'aszl\'o M\'erai, Harald Niederreiter
}

\begin{document}

 \maketitle

\begin{abstract}
In 2012, Diem introduced a new figure of merit for cryptographic
sequences called expansion complexity. In this paper, we slightly
modify this notion to obtain the so-called irreducible-expansion complexity which is more suitable for certain
applications. We analyze both, the classical and the modified expansion
complexity. Moreover, we also study the expansion complexity of the
explicit inversive congruential generator.
\end{abstract}

\textit{Key words and phases:} pseudorandom sequence, expansion complexity, inversive generator

\section{Introduction}

Sequences over finite fields which are generated by  a short linear recurrence relation
are considered  cryptographically weak. This observation leads to the
notion of \emph{linear complexity profile} of sequences, which is an
infinite sequence of nondecreasing integers such that the $N$th term
is the length of a shortest linear recurrence relation which generates
the first $N$ elements of the sequence.
The linear complexity profile is a measure for the unpredictability of
a sequence and thus its suitability in cryptography. A sequence with
small $N$th linear complexity (for a sufficiently large $N$) is
disastrous for cryptographic applications.
We recommend the interested reader to consult the survey of Meidl and
Winterhof~\cite{MW13}  and previous articles by
Niederreiter~\cite{Niederreiter03} and Winterhof~\cite{Winterhof10}.

Xing and Lam~\cite{XL99} gave a general construction of infinite
sequences over finite fields with optimal linear complexity. The
construction is based on functional expansion into expansion series.
Diem~\cite{di12} showed that this type of sequence can be efficiently
computed from a relatively short subsequence. This observation leads
to the \emph{expansion complexity}. For the connection between the
linear and expansion complexity we refer to the recent paper
\cite{MeNiWi2016+}.

In this paper we study the properties of this figure of merit for
sequences over finite fields. In Section~\ref{sec:expansion} we slightly modify the notion
of expansion complexity to obtain the so-called i(rreducible)-expansion complexity which is more suitable for certain applications. We analyze the properties of both the classical and the modified expansion
complexity. Then we study the expansion complexity of the explicit
inversive congruential generator in Section~\ref{sec:inversive}. We prove that this
sequence has optimal expansion complexity
and we give a lower bound on the expansion complexity if the sequence
is randomly shifted. We finish the paper with a summary of the results
in Section~\ref{sec:conclusions}.

\section{Expansion sequences and expansion complexity}
\label{sec:expansion}

For a sequence $\S=(s_i)_{i=0}^\infty$ over the finite field $\F_q$ of
$q$ elements, we define the \emph{generating function} $G(x)$ of $\S$
by
\begin{equation*}
G(x) = \sum_{i=0}^{\infty}s_ix^i,
\end{equation*}
viewed as a formal power series over $\F_q$.

A sequence $\S$ is called an \emph{expansion sequence} if its
generating function satisfies an algebraic equation
\begin{equation}\label{eq:h}
 h(x,G(x))=0
\end{equation}
for some nonzero $h(x,y)\in\F_q[x,y]$. Clearly, the polynomials
$h(x,y)\in\F_q[x,y]$ satisfying \eqref{eq:h} form an ideal in
$\F_q[x,y]$. This ideal is called the \emph{defining ideal} and it is
a principal ideal generated by an irreducible polynomial, see
\cite[Proposition~4]{di12}.

Expansion sequences  can be efficiently computed from a relatively
short subsequence via the generating polynomial of its defining
ideal~\cite[Section~5]{di12}.

\begin{prop}
\label{prop:expansion_seq}
Let $\S$ be an expansion sequence and let $h(x,y)$ be the generating
polynomial of its defining ideal.
The sequence $\S$ is uniquely determined by $h(x,y)$ and its initial
sequence of length $(\deg h)^2$.
Moreover, $h(x,y)$ can be computed in polynomial time (in $\log q
\cdot \deg h $) from an initial sequence of length $(\deg h)^2$.
\end{prop}

Based on Proposition \ref{prop:expansion_seq}, Diem~\cite{di12}
defined the $N$th expansion complexity in the following way.
For a positive integer $N$, the {\em $N$th expansion complexity}
$E_N=E_N({\cal S})$ is $E_N=0$ if $s_0=\ldots=s_{N-1}=0$ and otherwise
the least total degree of a nonzero polynomial $h(x,y)\in \F_q[x,y]$ with
\begin{equation}
\label{eqcon}
h(x,G(x))\equiv 0 \mod x^N.
\end{equation}
Note that $E_N$ depends only on the first $N$ terms of ${\cal S}$.
However, small expansion complexity does not imply high predictability
in the sense of Proposition~\ref{prop:expansion_seq}.

\begin{example}\label{ex:1}
 Let $\S$ be a sequence  over the finite field $\F_p$ ($p\geq 3$) with
 initial segment $\S=000001\dots$ and generating function $G(x)\equiv
 x^5 \mod x^6$. Then its 6th expansion complexity is $E_2(\S)=2$
 realized by the polynomial $h(x,y)=x\cdot y$. However, the first 4
 elements do not determine the whole initial segment with length 6.
\end{example}

In order to achieve the predictability of sequences in terms of
Proposition~\ref{prop:expansion_seq}, one needs to require that the
polynomial $h(x,y)$ satisfying \eqref{eqcon} is
\emph{irreducible}. This observation leads to  the
\emph{i(rreducible)-expansion complexity} of a sequence. Accordingly,
for a positive integer $N$, the {\em $N$th i-expansion complexity}
$E^*_N=E^*_N({\cal S})$ is $E_N^*=0$ if $s_0=\ldots=s_{N-1}=0$ and
otherwise the least total degree
of an  irreducible polynomial $h(x,y)\in \F_q[x,y]$ with
\eqref{eqcon}.

\begin{example}
 Let $\S$ be the sequence in Example \ref{ex:1}. Then its 6th
 i-expansion complexity is $E^*_6(\S)=5$ realized by the polynomial
 $h(x,y)=y-x^5$.
\end{example}

Clearly, for any sequence $\S$ we have
\[
  E_N^*(\S)\leq E_{N+1}^*(\S)
\]
and
\begin{equation}\label{eq:E-bound}
 E_N(\S)\leq E_N^*(\S)\leq \max\{1, N-1\}.
\end{equation}
The second inequality immediately gives a bound on the expansion
complexity. In the following theorem we give a stronger bound.

\begin{theorem}\label{thm:E-bound}
For any sequence
$\S$,
the expansion complexity $E_N(\S)$ satisfies the following inequality:
\begin{equation}\label{eq:abs_bound}
\binom{E_N(\S)+1}{2}\leq N.
\end{equation}
\end{theorem}

\begin{proof}
With an integer $d$, consider the set of monomials
\begin{equation*}
M(d) = \{x^{i}y^{j} | i+j\le d\}
\end{equation*}
of size $\#M(d)=\binom{d+2}{2}$.
For each monomial in that set,
$x^iy^j\in M(d)$, we substitute \(y=G(x)\) and reduce it modulo
\(x^N\) to obtain a
polynomial of degree  at most \(N-1\). The set of all polynomials of
degree less
than $N$ is a vector space over $\F_q$ of dimension $N$. Each of the
evaluations of the monomials in
$M(d)$ gives a polynomial in that space and there are
$\binom{d+2}{2}$ of these monomials, which means that they are linearly dependent
if there are more than $N$. Now we put $d=E_N(\S)-1$. If \eqref{eq:abs_bound} were not satisfied, then the argument just presented leads to a contradiction.
\end{proof}

It follows from \eqref{eq:abs_bound} that $E_N(\S)\leq
\sqrt{2N}$. On the other hand, for the i-expansion complexity, we have
$E_N^{*}(\S)\geq \sqrt{2N}$ for almost all sequences, it as will be shown in Theorem~\ref{thm:prob} below.

Let $\mu_q$ be the uniform probability measure on $\F_q$ which assigns
the measure $1/q$ to each element of $\F_q$. Let $\F_q^{\infty}$ be
the sequence space over $\F_q$
and let $\mu_q^{\infty}$ be the complete product probability measure
on $\F_q^{\infty}$ induced by $\mu_q$. We say that a property of
sequences ${\cal S} \in \F_q^{\infty}$
holds $\mu_q^{\infty}$-\emph{almost everywhere} if it holds for a set
of sequences ${\cal S}$ of $\mu_q^{\infty}$-measure $1$. We may view
such a property as a typical property of a random sequence over $\F_q$.
\begin{theorem}
  \label{thm:prob}
  We have
  \[
    \liminf_{N \to \infty} \, \frac{E^*_N({\cal S})}{\sqrt{2N}} \geq 1
    \qquad \mu_q^{\infty}\mbox{-almost everywhere}.
  \]
\end{theorem}
We remark, that Theorem~\ref{thm:prob} is the corrected form of \cite[Theorem~4]{MeNiWi2016+}.

\begin{proof}
  First we fix $\varepsilon$ with $0 < \varepsilon < 1$ and we put
\begin{equation*}
b_N= \lfloor (1-\varepsilon)\sqrt{2N} \rfloor \qquad \mbox{for } N=1,2,\ldots .
\end{equation*}
Then $b_N \ge 1$ for all sufficiently large $N$. For such $N$ put
\[
A_N=\{{\cal S} \in \F_q^{\infty}: E^*_N({\cal S}) \le b_N\}.
\]
Since $E^*_N({\cal S})$ depends only on the first $N$ terms of ${\cal
  S}$, the measure $\mu_q^{\infty}(A_N)$ is given by
\begin{equation} \label{equq}
\mu_q^{\infty}(A_N) =q^{-N} \cdot \# \{{\cal S} \in \F_q^N: E^*_N({\cal S}) \le b_N\}.
\end{equation}
An irreducible polynomial with degree $d$ can define at most $d$
expansion sequences (see \cite[p. 332]{di12}). Moreover, if two irreducible polynomials are
constant multiples of each other, they define the same sequences.

Let a polynomial $f(x,y)$ of degree $d$ be called \emph{normalized}
if in the  coefficient vector $(a_0,a_1,\dots, a_d)$ of the
homogeneous part with degree $d$ of $f$, i.e.,
\[
 a_0x^d+a_1x^{d-1}y+\dots+a_d y^d,
\]
the first nonzero element is 1.

Let $I_2(d)$ be the number of  \emph{normalized} irreducible
polynomials (with two variables) in $\F_q[x,y]$ of total degree $d$. Then by
\cite{carlitz63} we have
\[
 I_2(d)=\frac{1}{q-1}q^{\binom{d+2}{2}}+O\left(q^{\binom{d+1}{2}}\right).
\]
Thus
\begin{align}\label{eq:bound}
  \{{\cal S} \in \F_q^N: E^*_N({\cal S}) \le b_N\}\leq \sum_{d=1}^{b_N} d \cdot I_2(d)
  \ll  \sum_{d=1}^{b_N} d \cdot q^{\binom{d+2}{2}-1}\ll b_N q^{\binom{b_N+2}{2}-1}.
\end{align}
Thus it follows from \eqref{equq} and \eqref{eq:bound} that
$\mu_q^{\infty}(A_N)\leq q^{-\delta N}$ for some positive $\delta$ and
for all sufficiently large $N$. Therefore
$\sum_{N=1}^{\infty} \mu_q^{\infty}(A_N) < \infty$. Then the
Borel-Cantelli lemma (see \cite[Lemma~3.14]{Br}) shows that the set of all ${\cal S} \in
\F_q^{\infty}$ for which ${\cal S} \in A_N$ for infinitely many $N$
has $\mu_q^{\infty}$-measure $0$. In other words,
$\mu_q^{\infty}$-almost everywhere we have
${\cal S} \in A_N$ for at most finitely many $N$. It follows then from
the definition of $A_N$ that $\mu_q^{\infty}$-almost everywhere we
have
$$
E^*_N({\cal S}) > b_N > (1-\varepsilon) \sqrt{2N} -1
$$
for all sufficiently large $N$. Therefore $\mu_q^{\infty}$-almost
everywhere,
$$
\liminf_{N \to \infty} \, \frac{E^*_N({\cal S})}{\sqrt{2N}} \geq
1-\varepsilon.
$$
By applying this for $\varepsilon =1/r$ with $r=1,2,\ldots$ and noting
that the intersection of countably many sets of
$\mu_q^{\infty}$-measure $1$ has again $\mu_q^{\infty}$-measure $1$,
we obtain the result of the theorem.
\end{proof}

We finish this section showing that, for sequences having maximal
expansion complexity, we have $E^*_N({\cal S}) = E_{N}({\cal S})$.
\begin{theorem}
  \label{star}
  If the sequence $\mathcal{S}$ has maximal expansion complexity,
  i.e. if for $d\ge 1$, we have
\begin{equation*}
   E_N(\S)=d \quad \text{whenever} \quad \binom{d+1}{2}\leq N<\binom{d+2}{2},
 \end{equation*}
 then
 \begin{equation*}
   E^*_N(\S) = d' \quad \text{whenever} \quad \binom{d'+1}{2}+2\leq N<\binom{d'+2}{2},
 \end{equation*}
 for $d'\geq 6$.
\end{theorem}

\begin{proof}
Let $d\geq 6$ and assume that $\binom{d+1}{2}+2\leq N$. 
  We will show that if a polynomial $h(x,y)$ satisfies the congruence~\eqref{eqcon} with total degree equal to $d=E_{N}({\cal S})$, then
  it must be irreducible.
  We proceed proving the result by assuming the opposite, that is
  $h(x,y) = h_1(x,y)h_2(x,y)$ and  $d_1 = \deg h_1(x,y)$ and
  $d_2 = \deg h_2(x,y)$ positive. Then $h(x,y)$
  satisfies \eqref{eqcon} if and only if for nonnegative integers $N_1,
  N_2$ with $N = N_1 + N_2$,
  \begin{equation*}
    h_1(x,G(x))\equiv 0\mod x^{N_1},\quad h_2(x,G(x))\equiv 0\mod x^{N_2}.
  \end{equation*}
  Without loss of generality, we may suppose that  $N_1$ and $N_2$ are
  positive integers. We also suppose that
  $E_{N_1}({\cal S}) = d_1$ and $E_{N_2}({\cal S}) = d_2$
  Applying Theorem~\ref{thm:E-bound}, we obtain
  \begin{equation*}
    \binom{d_1+d_2 +1}{2}\le N_1+N_2 <
    \binom{d_1+2}{2}+\binom{d_2+2}{2}.
  \end{equation*}
  This implies by simple manipulation that
  \[
   (d_1-1)(d_2-1)\leq 2.
  \]
  If the last inequality  holds, then either $d_1 = 1$, or $d_2=1$ by the assumption $d_1+d_2=d\geq 6$.
  If $d_1=1$, then $N_1\le 2$ and, applying again
  Theorem~\ref{thm:E-bound},  implies
  \[
    \binom{d_2 +2}{2}-2\le N_2 < \binom{d_2 +2}{2} 
  \]
i.e.
  \[
    \binom{d +1}{2}\le N < \binom{d + 1}{2}+2,
  \]
a contradiction. We proceed similarly in the case $d_2=1$.
\end{proof}

% We want to mention that for any $d$ and $N = \binom{d +1}{2}$ appears
% polynomials $h(x,y)$ which are divisible by $x$. Also, for
% $N = \binom{d +1}{2}+1$ or $N= \binom{d +1}{2}+2$, there are polynomials satisfying
% equation~\eqref{eqcon} which are divisible by $y-s_0-s_1x$.

\section{Expansion complexity of the explicit inversive congruential generator}\label{sec:inversive}

The \emph{explicit inversive congruential generator} is defined in a prime field
$\mathbb{F}_p$ ($p\geq 3$) by
\begin{equation}\label{eq:explicit}
  s_n = n^{p-2} \bmod p \quad \text{for }  n=0,1\ldots.
\end{equation}
Clearly, this is a purely periodic sequence with least period
length $p$. We show that its expansion complexity is maximal in terms
of Theorem~\ref{thm:E-bound}.

\begin{theorem}
The explicit inversive generator $\S=(s_n)$ defined in~\eqref{eq:explicit} has
  maximal expansion complexity for all $N=2,\ldots, p-1$, i.e. we have
    \begin{equation}\label{eq:thm3}
    E_N(\S)=d \quad \text{whenever} \quad \binom{d+1}{2}\leq N<\binom{d+2}{2}.
  \end{equation}
\end{theorem}

By \eqref{eq:E-bound} and Theorem~\ref{star}, this result gives a
lower bound for
$E_N^{*}(\S)$ for $N\leq p-1$ which is in line with the asymptotic regime in Theorem~\ref{thm:prob}.

\begin{proof}
By Theorem~\ref{thm:E-bound}, we have $E_N(\S)\leq d$ if $N$ is in the
range~\eqref{eq:thm3}.
Thus it suffices to prove the lower bound $E_N(\S)\geq d$ for such
$N$. As the $N$th expansion complexity $E_N(\S)$ is a nondecreasing
function of $N$, it is enough to prove the result for integers $N$
having the form $N=\binom{d+1}{2}$  with some positive integer $d$.

We remark that the derivative $G'(x)$ of the generating function $G(x)$ of $\S$ satisfies
\begin{equation}\label{eq:inv-G}
 G'(x)=\left( \sum_{n=0}^{\infty}n^{p-2}x^n   \right)'
 =\sum_{\substack{0\leq n<\infty \\ p\nmid
     n+1}}x^n=\frac{1}{1-x}-x^{p-1}\frac{1}{1-x^p}.
\end{equation}

Now we prove the theorem by induction on $d$.
For $d=2$ $(N=3)$ the assertion follows from straightforward
computation. Next, we prove the theorem by contradiction. Assume that
there is a $d>2$ that does not satisfy the assertion. Let $d$ be the
smallest such integer. Then $E_{N-d}(\S)= \ldots = E_N(\S)= d-1$ with
$N=\binom{d+1}{2}$.

By recursion, we construct nonzero polynomials $f_i(x,y)\in\F_p[x,y]$
$(i=0,1\ldots, d-1)$ of total degree $d-1$ such that
  \begin{equation}
    \label{eq:fi}
    f_i(x,G(x))\equiv0\bmod x^{N-i}
  \end{equation}
  and
  \begin{equation}
    \label{eq:fi2}
    f_i(x,y)\text{ does not contain the terms }x^{d-1-\ell}y^{\ell},\
    0\le \ell<i.
  \end{equation}
  By assumption $E_N(\S)= d-1$, thus there is a nonzero polynomial
  $f(x,y)\in \F_p[x,y]$ of total degree  $d-1$ such that
  \begin{equation}
    \label{eq:fi0}
    f(x,G(x)) \equiv 0\bmod x^{N}.
  \end{equation}
  Put $f_0(x,y)=f(x,y)$. Now suppose that $f_i(x,y)$ has been constructed for some $0\leq i \leq d-2$. To construct the polynomial $f_{i+1}(x,y)$, we take the derivative of \eqref{eq:fi} with
  respect to $x$:
  \begin{equation}
    \label{eq:derivate1}
    \frac{\partial f_i}{\partial x}(x,G(x)) + \frac{\partial f_i}{\partial y}(x,G(x))
    G'(x) \equiv
    0 \bmod x^{N-i-1}.
  \end{equation}
 As
  \[
   G'(x)\equiv \frac{1}{1-x} \mod x^{p-1}
  \]
by~\eqref{eq:inv-G},  we obtain
  \begin{multline*}
    (1-x)\left (\frac{\partial f_i}{\partial x}(x,G(x)) +
      \frac{\partial f_i}{\partial y}(x,G(x))
    G'(x)\right ) \equiv \\
  (1-x)\frac{\partial f_i}{\partial x}(x,G(x)) + \frac{\partial
    f_i}{\partial y}(x,G(x)) \equiv 0\bmod x^{N-i-1}.
  \end{multline*}
  Put
  \begin{equation*}
    g_i(x,y) = (1-x)\frac{\partial f_i}{\partial x}(x,y) + \frac{\partial
    f_i}{\partial y}(x,y)\in\F_p[x,y].
  \end{equation*}
  Observe, that $g_i(x,y)$ and $f_i(x,y)$ have the same total
  degree. Indeed, if the total degree of $g_i(x,y)$ were strictly less than the total degree of $f_i(x,y)$, then we get a polynomial of total degree at mos $d-2$ satisfying \eqref{eq:fi} (with $i$ replaced by $i+1$), hence $E_{N-i-1}(\S)\leq d-2$, a contradiction. Moreover, the monomials of degree $d-1$  that appear in
  $g_i(x,y)$ must involve $x$ and appear in $f_i(x,y)$. If
  $f_i(x,y)=cg_i(x,y)$ for some nonzero $c\in\F_p$, then
  \[
   f_i(x,y)\equiv c^\ell \frac{\partial ^{\ell}f_i}{\partial
     ^{\ell}y}(x,y) \mod 1-x \quad \text{for all } \ell\geq 0.
  \]
  In particular, $(1-x)$ divides $f_i(x,y)$, so taking $f_i(x,y)/(1-x)$,
  we get a polynomial with total degree $d-2$ satisfying~\eqref{eq:fi},
  thus $E_{N-i}(\S)\leq d-2$, a contradiction.

  So, there must exist a nonzero linear combination $f_{i+1}(x,y)$ of
  $f_i(x,y)$ and $g_i(x,y)$ satisfying~\eqref{eq:fi} (with $i$ replaced by $i+1$)
  and~\eqref{eq:fi2}. If the total degree of $f_{i+1}(x,y)$
  were less than or equal to $d-2$, then $E_{N-i-1}(\S)\le d-2$, a
  contradiction.

  Finally, observe that if we construct $g_{d-1}(x,y)$ as above, then it does not contain the terms
  $x^{d-1-\ell}y^{\ell}$, $\ell = 0,\ldots d-2$, by construction, and
  also that it does not contain the term $y^{d-1}$. Thus, the total degree  of $g_{d-1}(x,y)$ is
  at most $d-2$. Moreover, it follows from~\eqref{eq:fi} for $i=d-1$,
  by the same argument as above,
  that
  \begin{equation*}
    g_{d-1}(x,G(x))\equiv 0\bmod x^{N-d},
  \end{equation*}
  thus $E_{N-d}(\S)\le d-2$, a contradiction.
\end{proof}

As a corollary, we obtain, for many different shifts of the
explicit inversive generator, a good lower bound on the expansion complexity.

\begin{cor}
  For any $d>0$ and all values of $1\le m< p$ but for $ (d-1)^2 \cdot \binom{d}{2}$ choices, the shifted
  explicit inversive generator $\mathcal{S}' = (s_{n+m})$ satisfies,
  \begin{equation*}
    E_N(\S')= d \quad \text{if} \quad \binom{d+1}{2}\leq N< \min\left\{ \binom{d+2}{2},p\right\}.
  \end{equation*}
\end{cor}

\begin{proof}
  We fix the value $N = \binom{d+1}{2}$ and take again the set of
  monomials
  \begin{equation*}
    M(d-1) = \{x^{i}y^{j} | i+j\le d-1\}.
  \end{equation*}
  Then we define the polynomial $G(x,m) =  \sum_{i=0}^{p-1}(i+m)^{p-2}x^i$ in the variables $m$ and $x$.
  For each monomial in $M(d)$, we can substitute $y=G(x,m)$ and
  reduce it modulo $x^N$ to obtain a polynomial of degree at most
  $N-1$ in the variable $x$.
  The set of all polynomials of degree in the variable $x$ less than $N$
  is a vector space   over the field of rational functions in the
  variable $m$ of dimension $N$. Each of the evaluations of the
  monomials gives a polynomial in that space, which can be seen as a
  vector of length $N$.

  All of the vectors can be written as rows of a matrix and
  $E_{N}(\S') = d$
  if and only if the determinant of this matrix is different from
  $0$. Multiply all the elements of this matrix by
  $\prod_{i=0}^{d-1}(m+i)^{d-1}$ and reduce them using that
  $(m+i)^p=(m+i)$, so the result is a matrix whose entries are
  polynomials in the variable $m$ and of degree less than $(d-1)^2$. The
  determinant is a polynomial of degree at most $(d-1)^2 \cdot \#M(d-1)$, which is not
  the zero polynomial because the determinant is different from zero
  for $m=0$. The number of roots of the determinant is at most $(d-1)^2 \cdot \#M(d-1)$,
  and this remark finishes the proof.
\end{proof}

%Domingo: Journal of Complexity requires a section which summarizes
%all the results of the paper.
\section{Conclusions}
\label{sec:conclusions}
In this paper, we have studied the expansion complexity and a slight
modification of this measure called i-expansion complexity. For the
expansion complexity, we have found an upper bound which answers
positively to a conjecture posed by M\'erai, Niederreiter, and
Winterhof~\cite{MeNiWi2016+}.

Regarding the i-expansion complexity, Theorem~\ref{thm:prob}
shows that its behavior is different  and it is expected that the
i-expansion is a stronger measure than the expansion
complexity. However, if the expansion complexity of the sequence is
maximal, then  by Theorem~\ref{star}, the i-expansion complexity is
essentially equal to the expansion complexity.

For the explicit inversive generator, we have shown that the expansion
complexity and the i-expansion complexity are maximal. Even if the
sequence is shifted randomly, it is expected that the expansion
complexity is quite large.

\subsection*{Acknowledgments}
The authors would like to thank Arne Winterhof for his helpful comments.

The research of the first author was supported by the Ministerio de Economia y Competitividad research project MTM2014-55421-P.
The second was partially supported by the Austrian Science Fund FWF Project F5511-N26 which is part of the Special Research Program "Quasi-Monte Carlo Methods: Theory and Applications".

D. G.-P.: Department of Mathematics, University of Cantabria, Santander 39005, Spain, 
\\
\textit{E-mail address:} \url{domingo.gomez@unican.es}

\smallskip

L. M.: Johann Radon Institute for Computational and Applied Mathematics,
Austrian Academy of Sciences, Altenberger Stra\ss e 69, A-4040 Linz, Austria
\\
\textit{E-mail address:} \url{laszlo.merai@oeaw.ac.at}

\smallskip

H. N.: Johann Radon Institute for Computational and Applied Mathematics,
Austrian Academy of Sciences, Altenberger Stra\ss e 69, A-4040 Linz, Austria\\
\textit{E-mail address:} \url{harald.niederreiter@oeaw.ac.at}

\end{document}